\documentclass[11pt]{amsart}
\usepackage{amsfonts,amssymb,amsthm}
\usepackage[numbers,sort&compress]{natbib}
\usepackage[leqno]{amsmath}
\usepackage[mathscr]{euscript}
\usepackage{graphicx} 
\usepackage{booktabs}
\usepackage{color}
\usepackage{mathrsfs}
\usepackage{paralist}
\usepackage{hyperref}
\usepackage{framed}
\usepackage{epstopdf}
\usepackage{algorithm}
\usepackage{algorithmic}
\usepackage{caption}
\newtheorem {theorem}{Theorem}[section]
\newtheorem {proposition}{Proposition}[section]

\newtheorem {lemma}{Lemma}[section]
\newtheorem {example}{Example}[section]
\newtheorem {definition}{Definition}[section]
\newtheorem {remark}{Remark}[section]
\newtheorem {assumption}{Assumption}[section]

\newcommand{\R}{{\mathbb R}}
\newcommand{\N}{\mathbb{N}}

\newcommand{\argmin}{\operatornamewithlimits{argmin}}

\makeatletter
\newcommand{\leqnomode}{\tagsleft@true}
\newcommand{\reqnomode}{\tagsleft@false}
\makeatother

\def\ees{{\accent"5E e}\kern-.385em\raise.2ex\hbox{\char'23}\kern-.08em}
\def\EES{{\accent"5E E}\kern-.5em\raise.8ex\hbox{\char'23 }}
\def\ow{o\kern-.42em\raise.82ex\hbox{
   \vrule width .12em height .0ex depth .075ex \kern-0.16em \char'56}\kern-.07em}
\def\OW{O\kern-.460em\raise1.36ex\hbox{
\vrule width .13em height .0ex depth .075ex \kern-0.16em \char'56}\kern-.07em}

\title{A parameter-free approach for solving SOS-convex semi-algebraic fractional programs}

\author{Chengmiao Yang}
\address[Chengmiao Yang]{Academy for Advanced Interdisciplinary Studies, Northeast Normal University, Changchun 130024, Jilin Province, China}
\email{cmyang@nenu.edu.cn}

\author{Liguo Jiao$^{\dag}$}
\address[Liguo Jiao]{Academy for Advanced Interdisciplinary Studies, Northeast Normal University, Changchun 130024, Jilin Province, China; Shanghai Zhangjiang Institute of Mathematics, Shanghai 201203, China}
\email{jiaolg356@nenu.edu.cn; hanchezi@163.com}

\author{Jae Hyoung Lee$^{\dag}$}
\address[Jae Hyoung Lee]{Department of Applied Mathematics, Pukyong National University, Busan, 48513, Korea}
\email{mc7558@naver.com}

\thanks{$^{\dag}$Corresponding Authors}

\date{\today}

\begin{document}

\begin{abstract}
In this paper, we study a class of nonsmooth fractional programs {\rm (FP, for short)} with SOS-convex semi-algebraic functions.
Under suitable assumptions, we derive a strong duality result between the problem (FP) and its semidefinite programming (SDP) relaxations.
Remarkably, we extract an optimal solution of the problem (FP) by solving one and only one associated SDP problem.
Numerical examples are also given.
\end{abstract}
\maketitle

\section{Introduction}\label{Sec:1}
Consider the following standard fractional programming problem,
\begin{align}\label{FPP}
\overline \gamma := \min\limits_{x\in S} \ \frac{f(x)}{g(x)},
\end{align}
where $S$ is a nonempty subset of $\mathbb{R}^n,$ the function $f$ is non-negative and the function
$g$ is positive over $S.$
A remarkable feature of the fractional programming problem~\eqref{FPP} is that its objective function is generally not convex, even under rather restrictive convexity/concavity assumptions.
It is also worth mentioning that a great of ground-breaking results as well as applications on fractional programs were contributed by Schaible and his collaborators; see, e.g., \cite{Barros1996,Cambini1993,Schaible1981,Schaible2003} and the references therein.
Actually, the ratio to be optimized often describes some kind of efficiency of a system like economics, management science and other extremum problems; see, e.g., \cite{Stancu-Minasian1997,Bajalinov2003}.
In particular, a fractional program also occurs sometimes indirectly in modeling where initially no ratio is involved, see \cite{Schaible1995} for a comprehensive survey.

As is well-known, one of the most famous methods for solving the problem~\eqref{FPP} is the parameter approach due to Dinkelbach (see, e.g., \cite{Dinkelbach1967}),
which relates it to the following {\it non-fractional} problem,
\begin{align}\label{NFPP}
\min\limits_{x\in S} \ \left\{f(x) - \overline \gamma g(x)\right\}.
\end{align}
The Dinkelbach lemma states that ``the problem~\eqref{FPP} has an optimal solution $\overline x \in S,$
if and only if $\overline x$ is an optimal solution of the problem~\eqref{NFPP} and its optimal value is equal to zero with $\overline \gamma = \frac{f(\overline x)}{g(\overline x)}.$"
This tells us that finding an optimal solution to the problem~\eqref{FPP} can be achieved
by solving the problem~\eqref{NFPP}, a {\it non-fractional} one.
However, one drawback of the Dinkelbach procedure is that this can be proceeded only when the optimal value $\overline \gamma$ of the problem~\eqref{FPP} is known in advance.
Nevertheless, if we take some special structures of the functions $f$ and $g$ into account, then the optimal value $\overline \gamma$ can be obtained, and by using the Dinkelbach procedure, the problem~\eqref{FPP} can be solved; see, e.g., \cite{Guo2021,Guo2024,Lee2018,Jiao2019b}.
Particularly, when the functions $f$ and $g$ are SOS-convex polynomials (see Definition~\ref{SOS-convexity}), one can obtain the optimal value $\overline \gamma$ of the problem~\eqref{FPP} by solving a related SDP problem of the problem~\eqref{FPP}, and then putting the obtained $\overline \gamma$ into the Dinkelbach procedure, one can get the optimal solution of the problem~\eqref{FPP} by solving another related SDP problem of the problem~\eqref{NFPP}; see, e.g., \cite{Lee2018}.
Recently, along with the Dinkelbach method, Jiao and Lee \cite{Jiao2019b} studied a class of nonsmooth fractional programs with the functions $f$ and $g$ consisting of sums of an SOS-convex polynomial function and a support function, and they solved this kind of problems by solving also two related SDP problems, one is for getting the optimal value $\overline \gamma,$ another is for obtaining the optimal solution $\overline x.$
Some more recent important works on fractional programs are contributed by Bot et al. \cite{Bot2021,Bot2023}.

In this paper, we aim at solving a new class of nonsmooth fractional programs with SOS-convex semi-algebraic functions (see Definition~\ref{SOS-SA}).
Such a class of nonsmooth convex functions, which extends the notion of SOS-convex polynomials, was introduced by Chieu et al. \cite{Chieu2018}.
It is worth noting that, under commonly used strict feasibility conditions, the optimal value and optimal solutions of SOS-convex semi-algebraic programs can be found by solving a single SDP problem.
Among other things, we make the following contributions to fractional programs.
\begin{itemize}
\item[{\rm (i)}] The model in consideration is a rather rich framework that includes at least quadratic fractional problems \cite{Nguyen2016}, SOS-convex polynomial fractional problems~\cite[Section 4]{Jeyakumar2012} and \cite{Lee2018}, and fractional problems with sums of an SOS-convex polynomial and a support function \cite{Jiao2019b} as special cases.
\item[{\rm (ii)}] Different to the Dinkelbach method, we propose a parameter-free approach for solving the class of nonsmooth fractional programs with SOS-convex semi-algebraic functions, and this parameter-free approach allows us to extract an optimal solution (and the optimal value) by solving one and only one associated SDP problem.
\end{itemize}

The outline of the paper is organized as follows. Section~\ref{sect:2} presents some notations and preliminaries. Section~\ref{sect:3} gives the main results of the paper on finding an optimal solution of the fractional problems in consideration.  Finally, conclusions are given in Sect.~\ref{sect:4}.

\section{Notation and preliminaries}\label{sect:2}

Throughout the paper, the space $\mathbb{R}^{n}$ is equipped with the standard Euclidean norm, i.e., $\| x \|=\sqrt{x_1^2+\cdots+x_n^2}$ for all $x := (x_{1}, x_{2}, \ldots, x_{n}) \in \R^n.$
We suppose $1 \leq n \in \mathbb{N}$, where $\mathbb{N}$ stands for the set of non-negative integers.
The non-negative orthant of $\mathbb{R}^{n}$ is denoted by
$$\mathbb{R}_{+}^{n}:=\left\{x \in \R^n \colon x_i \geq 0, \ i = 1, \ldots, n\right\}.$$

\subsection{Convex analysis}
For an extended real-valued function $f$ on $\mathbb{R}^{n},$ $f$ is said to be {\it proper} if for all $x\in \mathbb{R}^{n},$ $f(x)>-\infty$ and there exists $\hat{x}\in \mathbb{R}^{n}$ such that $f(\hat{x})\in\mathbb{R}.$
We denote its domain and epigraph of $f$ by ${\rm dom\,} f:=\{x\in \mathbb{R}^{n}: f(x)<+\infty\}$ and ${\rm epi\,}f:=\{(x,r)\in \mathbb{R}^{n}\times\mathbb{R}: f(x)\leq r\},$ respectively.
We say a function $f$ is lower semicontinuous (l.s.c.) if $\liminf_{y\to x}f(y)\geq f(x)$ for all $x\in\mathbb{R}^{n} ;$ in other words, ${\rm epi\,}f$ is closed in $\mathbb{R}^{n+1};$ see, e.g., \cite[Theorem 7.1]{Rockafellar1970}.
A function $f\colon\mathbb{R}^{n}\to\mathbb{R}\cup\{+\infty\}$ is said to be {\it convex} if ${\rm epi\,}f$ is a convex set.
As usual, for any proper convex function $f$ on $\mathbb{R}^{n},$
its conjugate function $f^*\colon\mathbb{R}^{n} \to \mathbb{R}\cup \{+\infty\}$  is defined by $f^*(x^*)=\sup \left\{\langle x^*,x\rangle-f(x) : x\in\mathbb{R}^{n}\right\}$ for all $x^*\in\mathbb{R}^{n}.$

For a given set $A\subset \mathbb{R}^{n},$ we denote the closure and the convex hull generated by $A$ by ${\rm cl\,}A$ and ${\rm conv\,}A,$ respectively.
The indicator function $\delta_A$ is defined by
\begin{equation*}
\delta_A(x):=\left\{
\begin{array}{@{\,}ll}
0, & x\in A,\\
+\infty, & {\rm otherwise}.
\end{array}
\right.
\end{equation*}
Note that, if $A$ is a convex set, then the indicator function $\delta_A$ is convex; if $A$ is closed, then $\delta_A$ is l.s.c..
\begin{lemma}{\rm\cite{Jeyakumar2006}}\label{lemma1}
Let $f\colon\mathbb{R}^{n} \to \mathbb{R}\cup \{+\infty\}$ and $g\colon \mathbb{R}^{n} \to \mathbb{R}\cup \{+\infty\}$ be proper l.s.c. convex functions. If ${\rm dom\,}f\cap {\rm dom\,}g\neq\emptyset,$ then
\begin{equation*}
{\rm epi}\left(f+g\right)^*={\rm cl}\left({\rm epi\,}f^* + {\rm epi\,} g^*\right).
\end{equation*}
Moreover$,$ if one of the functions $f$ and $g$ is continuous$,$ then
\begin{equation*}
{\rm epi}\left(f+g\right)^*={\rm epi\,}f^* + {\rm epi\,}g^*.
\end{equation*}
\end{lemma}

\begin{lemma}{\rm\cite{Jeyakumar2003,Li2008}}\label{lemma2}
Let $g_{i}\colon  \mathbb{R}^{n} \to \mathbb{R}\cup \{ + \infty \},$ $i\in I,$ be a proper l.s.c. convex function$,$ where $I$ is an arbitrary index set.
Suppose that there exists $x_{0}\in\mathbb{R}^{n}$ such that $\sup_{i\in I}g_{i}(x_{0})<+\infty.$ Then
\begin{equation*}
{\rm epi}\left(\sup_{i\in I}g_{i}\right)^*={\rm cl}\left({\rm conv\,}\bigcup_{i\in I}{\rm epi\,}g_{i}^*\right).
\end{equation*}
\end{lemma}

\subsection{Real polynomials}
The space of all real polynomials in the variable $x$ is denoted by $\mathbb{R}[x];$
moreover, the space of all real polynomials in the variable $x$ with degree at most $d$ is denoted by $\mathbb{R}[x]_{d}.$
The degree of a polynomial $f$ is denoted by $\deg f.$
We say that a real polynomial $f$ is sum of squares, if there exist real polynomials $q_{\ell},$ $\ell=1,\ldots,s,$ such that $f =\sum_{\ell=1}^{s}q_{\ell}^{2}.$
The set consisting of all sum of squares of real polynomials with degree at most $d$ is denoted by $\Sigma[x]_{d}^{2}.$
For a multi-index $\alpha\in\mathbb{N}^{n},$ let $|\alpha|:=\sum_{i=1}^n\alpha_{i},$ and let $\mathbb{N}^n_{d}:= \{\alpha \in \mathbb{N}^n :  |\alpha|\leq d\}.$
$x^\alpha$ denotes the monomial $x_{1}^{\alpha_{1}}\cdots x_{n}^{\alpha_n}.$
The canonical basis of $\mathbb{R}[x]_{d}$ is denoted by
\begin{equation}\label{cano_basis}
v_{d}(x):=(x^\alpha)_{\alpha\in\mathbb{N}^n_{d}}=(1,x_{1},\ldots,x_{n},x_{1}^{2},x_{1}x_{2},\ldots,x_{n}^{2},\ldots,x_{1}^{d},\ldots,x_{n}^{d})^{T},
\end{equation}
which has dimension $s(d):=\left( \substack{ n+d \\ n }\right).$
Given an $s(2d)$-vector $y:= (y_{\alpha})_{\alpha\in\mathbb{N}^{n}_{2d}}$,
let ${\rm M}_{d}(y)$ be the moment matrix of dimension $s(d),$ with rows and columns labeled by \eqref{cano_basis}.
For example, for $n=2$ and $d=1,$
\begin{equation*}
y=(y_{\alpha})_{\alpha\in\mathbb{N}^{2}_{2}}=(y_{00},y_{10},y_{01},y_{20},y_{11},y_{02})^T \ \textrm { and } \ {\rm M}_{1}(y)=\left(
         \begin{array}{ccc}
           y_{00} & y_{10} & y_{01} \\
           y_{10} & y_{20} & y_{11} \\
           y_{01} & y_{11} & y_{02} \\
         \end{array}
       \right).
\end{equation*}
In what follows, for convenience, we denote ${\rm M}_{d}(y):=\sum_{\alpha\in\mathbb{N}^{n}_{2d}}y_{\alpha} \mathcal{B}_{\alpha},$
where for each $\alpha\in\mathbb{N}^{n}_{2d},$ $\mathcal{B}_{\alpha}$ is a symmetric matrix of order $s(d),$ which is understood from the definition of ${\rm M}_{d}(y).$

Let $S^{n}$ be the set of $n\times n$ symmetric matrices.
The notion $\succeq$ denotes the L\"{o}wner partial order of $S^{n},$ that is, for $X,Y\in S^{n},$ $X\succeq Y$ if and only if $X-Y$ is a positive semidefinite matrix.
Let $S_{+}^{n}$ be the set of $n\times n$ symmetric positive semidefinite matrices.
Also, $X\succ0$ means that $X$ is a positive definite matrix.
For $X, Y\in S^{n},$ $\langle X, Y\rangle:={\rm tr}(XY),$ where ``${\rm tr}$'' denotes the trace of a matrix.
The gradient and the Hessian of a polynomial $f\in \mathbb{R}[x]$ at a point $\overline x$ are denoted by $\nabla f(\overline x)$ and $\nabla^{2}f(\overline x),$ respectively.

The following lemma, which plays a key role for our main result in the paper, shows a useful existence result for optimal solutions of convex polynomial programs.
\begin{lemma}\label{lemma3} {\rm \cite{Belousov2002}}
Let $f, g_{1},\ldots, g_{m}$ be convex polynomials on $\mathbb{R}^{n}.$
Let $K := \{x \in \mathbb{R}^{n} : g_{i}(x) \leq0, \ i = 1,\ldots,m\}\neq\emptyset.$
Suppose that $\inf_{x\in K}f(x) >-\infty.$
Then$,$ $\argmin_{x\in K}f(x)\neq\emptyset.$
\end{lemma}

Now, we recall the following result that shows how to confirm whether a polynomial can be written as a sum of squares via semidefinite programming.
\begin{proposition}\label{proposition1}{\rm \cite{Lasserre2009book}}
A polynomial $f\in\mathbb{R}[x]_{2d}$ has a sum of squares decomposition if and only if there exists $Q\in S^{s(d)}_{+}$ such that $f(x)=\langle v_{d}(x)v_{d}(x)^{T},Q\rangle$ for all $x\in\mathbb{R}^{n}.$
\end{proposition}
By letting $v_{d}(x)v_{d}(x)^{T}:=\sum_{\alpha\in\mathbb{N}_{2d}^n}x^{\alpha} \mathcal{B}_{\alpha},$
we have $f(x)=\sum_{\alpha\in\mathbb{N}^{n}_{2d}}f_{\alpha} x^{\alpha}$ is a sum of squares if and only if we can solve the following semidefinite feasibility problem \cite{Lasserre2009book}:
\begin{equation*}
\text{Find } Q\in S^{s(d)}_{+} \text{ such that } \langle \mathcal{B}_{\alpha},Q\rangle=f_{\alpha}, \ \forall \alpha\in\mathbb{N}_{2d}^n.
\end{equation*}

\medskip
Below, let us recall a very interesting subclass of convex polynomials in $\mathbb{R}[x]$ introduced by Helton and Nie~\cite{Helton2010} (see also \cite{Ahmadi2012,Ahmadi2013}).
\begin{definition}\label{SOS-convexity}{\rm \cite{Ahmadi2012,Ahmadi2013,Helton2010} A real polynomial $f$ on $\mathbb{R}^{n}$ is called {\it SOS-convex} if there exists a matrix polynomial $F(x)$ such that $\nabla^{2}f(x)=F(x)F(x)^{T}.$}
\end{definition}
It is worth noting that an SOS-convex polynomial is convex; but the converse is not true, which means that there exists a convex polynomial that is not SOS-convex \cite{Ahmadi2012,Ahmadi2013}.
Besides, SOS-convex polynomials enjoy many important properties, we list the following two of them for convenience. For more properties, see \cite{Helton2010}.
\begin{lemma}\label{lemma4} {\rm \cite[Lemma 8]{Helton2010}}  Let $f$ be an SOS-convex polynomial.
Suppose that there exists $\overline{x}\in \mathbb{R}^{n}$ such that $f(\overline{x}) = 0$ and $\nabla f(\overline{x}) = 0.$
Then$,$ $f$ is a sum of squares polynomial.
\end{lemma}

\begin{lemma}\label{lemma5}{\rm \cite{Lasserre2014}}
Let $f\in\mathbb{R}[x]_{2d}$ be SOS-convex$,$
and let $y=(y_{\alpha})_{\alpha\in\mathbb{N}_{2d}^n}$ satisfy $y_{0}=1$ and $\sum_{\alpha\in\mathbb{N}_{2d}^n}y_{\alpha} \mathcal{B}_{\alpha}\succeq0.$
Let $L_y \colon \mathbb{R}[x]\to\mathbb{R}$ be a linear functional defined by $L_y(f):=\sum_{\alpha\in\mathbb{N}_{2d}^n} f_{\alpha} y_{\alpha},$ where $f(x)=\sum_{\alpha\in\mathbb{N}_{2d}^n} f_{\alpha} x^{\alpha}.$ Then
\begin{equation*}
L_y(f)\geq f(L_y(x)),
\end{equation*}
where $L_y(x):=(L_y(x_{1}),\ldots,L_y(x_{n})).$
\end{lemma}

Finally, we close the section by recalling the notion of SOS-convex semi-algebraic functions, which was introduced and studied systematically by Chieu et al. \cite{Chieu2018}.
\begin{definition}[SOS-convex semi-algebraic functions]\label{SOS-SA}{\rm
A real function $f$ is called {\it SOS-convex semi-algebraic} on $\mathbb{R}^{n}$ if it admits a representation
\begin{equation*}
f(x) := \sup_{y\in\Omega}\bigg\{h_{0}(x) +\sum^s_{j=1}y_jh_j(x)\bigg\},
\end{equation*}
where
\begin{itemize}
  \item[{\rm (a)}] each $h_j,$ $j = 0, 1,\ldots,s,$ is a polynomial, and for each $y\in\Omega,$ $h_{0}+\sum^{s}_{j=1}y_jh_j$ is an SOS-convex polynomial on $\mathbb{R}^{n};$
  \item[{\rm (b)}] $\Omega$ is a nonempty compact semidefinite program representable set given by
\end{itemize}
\begin{equation*}
 \Omega:= \left\{y\in \mathbb{R}^s : \exists\, z\in \mathbb{R}^{p} \textrm{ s.t. } A_{0} +\sum^{s}_{j=1}y_{j}A_{j} +\sum^{p}_{\ell=1}z_{\ell}B_{\ell} \succeq 0\right\},
 \end{equation*}
with $A_{j},$ $j=0,1,\ldots,s,$ and $B_{\ell},$ $\ell=1,\ldots,p,$ being $t\times t$ symmetric matrices.}
\end{definition}


\section{Main Results}\label{sect:3}
In this section, we focus on the study of solving a class of fractional programs with SOS-convex semi-algebraic functions.
\leqnomode
\begin{align}\label{FP}
\min\limits_{x\in\mathbb{R}^{n}} \left\{ \frac{f_{m+1}(x)}{-f_{m+2}(x)} \colon f_{i}(x)\leq0, \ i=1,\ldots,m \right\},  \tag{${\rm FP}$}
\end{align}
where each $f_{i},$ $i=1,\ldots,m+2,$ is SOS-convex semi-algebraic function in the form
\begin{equation}\label{function f_{i}}
f_{i}(x) := \sup_{\left(y_{1}^{i},\ldots,y_{s_{i}}^{i}\right) \in \Omega_{i}}\left\{h_{0}^{i}(x) +\sum^{s_{i}}_{j=1}y_{j}^{i}h_{j}^{i}(x)\right\},
\end{equation}
such that
\begin{itemize}
  \item[(a)] each $h_{j}^{i}, j=0, \ldots, s_{i}, i=1,2, \ldots, m+2,$ is a polynomial with degree at most $2d,$ and for each $y^{i}:=(y_{1}^{i},\ldots,y_{s_{i}}^{i})\in\Omega_{i},$ $h_{0}^{i}+\sum^{s_{i}}_{j=1}y_{j}^{i}h_{j}^{i}$ is an SOS-convex polynomial on $\mathbb{R}^{n};$
  \item[(b)] $\Omega_{i}, i=1,2, \ldots, m+2,$ is a nonempty compact semidefinite program representable set given by
\end{itemize}
\begin{equation*}
 \Omega_{i}:= \left\{(y_{1}^{i},\ldots,y_{s_{i}}^{i})\in \mathbb{R}^{s_{i}} : \exists z^{i}\in \mathbb{R}^{p_{i}} \textrm{ s.t. } A_{0}^{i} +\sum^{s_{i}}_{j=1}y_{j}^{i}A_{j}^{i} +\sum^{p_{i}}_{\ell=1}z_{\ell}^{i}B_{\ell}^{i} \succeq 0\right\},
\end{equation*}
here, $A_{j}^{i},$ $j=0,1,\ldots,s_{i},$ and $B_{\ell}^{i},$ $\ell=1,\ldots,p_{i},$ are $t_{i}\times t_{i}$ symmetric matrices.
Moreover, we assume that $f_{m+1}(x)\geq 0$ and $-f_{m+2}(x)>0$ for all $x \in K,$ where $K$ is the feasible set of the problem~\eqref{FP}, defined by
$$K:=\{x\in\mathbb{R}^{n}: f_{i}(x)\leq0, \ i=1,\ldots,m\},$$
which is assumed to be nonempty.
Let $I:=\{1,2,\ldots,m+2\}.$

\begin{remark}
{\rm
We list the following frequently occurring objectives:
maximization of return/risk; minimization of cost/time; minimization of input/output.
Since the class of SOS-convex semi-algebraic functions contains many common nonsmooth convex functions; see, for example~\cite[Examples 3.1 \& 3.2]{Chieu2018}, thus
if the considered data (return, risk, input, output, etc.) is SOS-convex semi-algebraic function, then it can be exactly formulated in the form of the model~\eqref{FP}.
}\end{remark}

To proceed, by denoting the set
\begin{align}\label{convexcone}
\mathcal{C}:=\bigcup_{y^{i}\in\Omega_{i},\lambda_{i}\geq0}{\rm epi}\left(\sum_{i=1}^m\lambda_{i}\left(h_{0}^{i}+\sum_{j=1}^{s_{i}}y^{i}_{j}h^{i}_{j}\right)\right)^*,
\end{align}
which is a convex cone (see, e.g., \cite[Proposition~2.3]{Jeyakumar2010}),
we first give the following Farkas-type lemma, which plays an important role in deriving our results.
For completeness, we give a short proof.
\begin{lemma}\label{lemma6}
Let $f_{i}\colon\mathbb{R}^{n}\to\mathbb{R},$ $i\in I,$ be convex functions defined as in \eqref{function f_{i}}$,$ and let $\gamma\in\mathbb{R}_{+}.$
Assume that the set $K:=\{x\in\mathbb{R}^{n}: f_{i}(x)\leq0, \ i=1,\ldots,m\}$ is nonempty$.$
Then the following statements are equivalent$:$
\begin{itemize}
\item[{\rm (i)}] $K\subseteq \left\{x\in\mathbb{R}^{n} \colon f_{m+1}(x)+\gamma f_{m+2}(x)\geq0 \right\};$
\item[{\rm (ii)}] $(0,0)\in {\rm epi\,}\left(f_{m+1}+\gamma f_{m+2}\right)^*+{\rm cl\,}\bigcup\limits_{y^{i}\in\Omega_{i},\lambda_{i}\geq0}{\rm epi}\left(\sum\limits_{i=1}^m\lambda_{i}\left(h_{0}^{i}+\sum\limits_{j=1}^{s_{i}}y^{i}_{j}h^{i}_{j}\right)\right)^*.$
\end{itemize}
\end{lemma}

\begin{proof}
The statement ${\rm (i)}$  is equivalent to $\inf\limits_{x\in \mathbb{R}^{n}} \left\{f_{m+1}(x)+\gamma f_{m+2}(x)+\delta_{K}(x)\right\}\geq 0.$
Since $\langle 0,x\rangle-\left(f_{m+1}+\gamma f_{m+2}+\delta_{K}\right)(x)\leq 0$  for all $x\in\mathbb{R}^{n},$ we get $(f_{m+1}+\gamma f_{m+2}+\delta_{K})^*(0)\leq 0,$
equivalently, $(0,0)\in{\rm epi}(f_{m+1}+\gamma f_{m+2}+\delta_{K})^*.$
Since the function $f_{m+1}+\gamma f_{m+2}$ is continuous, by Lemma~\ref{lemma1}, we obtain
\begin{equation*}
(0,0)\in{\rm epi}(f_{m+1}+\gamma f_{m+2}+\delta_{K})^* = {\rm epi\,}(f_{m+1}+\gamma f_{m+2})^*+{\rm epi\,}\delta_{K}^*.
\end{equation*}
Note that
\begin{equation*}
\delta_{K}(x)=\sup_{y^{i}\in\Omega_{i},\lambda_{i}\geq0}\sum_{i=1}^m\lambda_{i}\left(h_{0}^{i}(x) +\sum^{s_{i}}_{j=1}y_{j}^{i}h_{j}^{i}(x)\right).
\end{equation*}
It follows from Lemma~\ref{lemma2} that
\begin{align}\label{farkas_lemma_rel2}
{\rm epi\,}\delta_{K}^*=&\,{\rm epi}\left(\sup_{y^{i}\in\Omega_{i},\lambda_{i}\geq0}\sum_{i=1}^m\lambda_{i}\left(h_{0}^{i}+\sum^{s_{i}}_{j=1}y_{j}^{i}h_{j}^{i}\right)\right)^*\notag\\
=&\, {\rm cl}\left({\rm conv}\bigcup_{y^{i}\in\Omega_{i},\lambda_{i}\geq0}{\rm epi}\left(\sum_{i=1}^m\lambda_{i}\left(h_{0}^{i} +\sum^{s_{i}}_{j=1}y_{j}^{i}h_{j}^{i}\right)\right)^*\right).
\end{align}
Moreover, since $\bigcup\limits_{y^{i}\in\Omega_{i},\lambda_{i}\geq0}{\rm epi}\left(\sum\limits_{i=1}^m\lambda_{i}\left(h_{0}^{i} +\sum\limits^{s_{i}}_{j=1}y_{j}^{i}h_{j}^{i}\right)\right)^*$ is a convex cone (see, e.g., \cite[Proposition~2.3]{Jeyakumar2010}), from \eqref{farkas_lemma_rel2}, we get the desired result.
\end{proof}

\subsection{Duality for the fractional programs with SOS-convex semi-algebraic functions}
Motivated by \cite[section 4]{Jeyakumar2014}, we now formulate the following SDP relaxation dual problem for the problem~\eqref{FP}.
\leqnomode
\begin{align}\label{SDD}
\sup\limits_{\substack{\lambda_{j}^{i},z_{\ell}^{i}},X} \ \ \ & \lambda_{0}^{m+2}\tag{$\widehat{\rm Q}$}\\
\textrm{s.t.}\quad   &\sum_{i=1}^{m+2}\left(\lambda_{0}^{i}\left(h_{0}^{i}\right)_{\alpha} +\sum^{s_{i}}_{j=1}\lambda_{j}^{i}\left(h_{j}^{i}\right)_{\alpha}\right)= \left\langle \mathcal{B}_{\alpha}, X \right\rangle, \ \alpha\in\mathbb{N}^{n}_{2d},\notag\\
&\lambda_{0}^{i}A_{0}^{i} +\sum^{s_{i}}_{j=1}\lambda_{j}^{i}A_{j}^{i} +\sum^{p_{i}}_{\ell=1}z_{\ell}^{i}B_{\ell}^{i}\succeq0,\ i=1,\ldots,m+2,\notag\\
&\lambda_{0}^{m+1}=1, \ \lambda_{0}^{i}\geq0, \ i=1,\ldots,m,m+2, \notag\\
&\lambda_{j}^{i}\in\mathbb{R}, \ z_{\ell}^{i}\in\mathbb{R}, \  i\in I, \ j=1,\ldots, s_{i}, \ \ell=1,\ldots,p_{i}.\notag
\end{align}
It is worth noting that the problem \eqref{SDD} is an SDP problem, and it can be efficiently solved via interior point methods.
Below, we propose a strong duality theorem between the problem~\eqref{FP} and its dual problem~\eqref{SDD}. Before that, we also need the following assumption.
\begin{assumption}\label{assumption1}
The convex cone $\mathcal{C}$ as in \eqref{convexcone} is closed.
\end{assumption}
\begin{remark}{\rm
It is worth mentioning that if the {\it Slater-type condition} holds for the problem~\eqref{FP}, that is, there exists $\hat{x}\in\mathbb{R}^{n}$ such that
$f_{i}(\hat{x}) < 0, \ i=1,\ldots,m,$
then the convex cone $\mathcal{C}$ is closed; see, e.g., \cite[Proposition~3.2]{Jeyakumar2010}.
}
\end{remark}

\begin{theorem}[\textbf{strong duality theorem}]\label{thm1}
Suppose that Assumption~\ref{assumption1} holds.
Then
\begin{equation*}
\inf\eqref{FP}=\max\eqref{SDD}.
\end{equation*}
\end{theorem}
\begin{proof}
We first show that $\inf\eqref{FP}\leq \max\eqref{SDD}.$
Let $\overline{\gamma}:=\inf\eqref{FP}\in\mathbb{R}_{+}.$
Then we have $f_{m+1}(x)+\overline{\gamma} f_{m+1}(x)\geq0$ for all $x\in K.$
It follows from Assumption~\ref{assumption1} and Lemma~\ref{lemma6} that there exist $\overline{\lambda}_{i}\geq0$ and $\overline{y}^{i}\in\Omega_{i},$ $i=1,\ldots,m,$ such that
\begin{equation*}
(0,0)\in {\rm epi\,}(f_{m+1}+\overline{\gamma} f_{m+2})^*+{\rm epi}\left(\sum_{i=1}^m\overline{\lambda}_{i}\left(h_{0}^{i}+\sum_{j=1}^{s_{i}}\overline{y}^{i}_{j}h^{i}_{j}\right)\right)^*.
\end{equation*}
Then, there exist $(\xi,\alpha)\in{\rm epi\,}(f_{m+1}+\overline{\gamma} f_{m+2})^*$ and $(\zeta,\beta)\in{\rm epi}(\sum_{i=1}^m\overline{\lambda}_{i}(h_{0}^{i}+\sum_{j=1}^{s_{i}}\overline{y}^{i}_{j}h^{i}_{j}))^*$ such that
$(0,0)=(\xi,\alpha)+(\zeta,\beta),$
and so, for each $x\in\mathbb{R}^{n},$ we have
\begin{align*}
&\langle \xi,x\rangle-(f_{m+1}(x)+\overline{\gamma} f_{m+2}(x))+\langle \zeta,x\rangle-\sum_{i=1}^m\overline{\lambda}_{i}\left(h_{0}^{i}(x)+\sum_{j=1}^{s_{i}}\overline{y}^{i}_{j}h^{i}_{j}(x)\right)\\
\leq\ &(f_{m+1}+\overline{\gamma} f_{m+2})^*(\xi)+\left(\sum_{i=1}^m\overline{\lambda}_{i}\left(h_{0}^{i}+\sum_{j=1}^{s_{i}}\overline{y}^{i}_{j}h^{i}_{j}\right)\right)^*(\zeta)\\
\leq\ & \alpha+\beta=0,
\end{align*}
i.e., for each $x\in \mathbb{R}^{n},$
\begin{align}\label{thm1rel3}
f_{m+1}(x)+\overline{\gamma} f_{m+2}(x)+\sum_{i=1}^m\overline{\lambda}_{i}\left(h_{0}^{i}(x)+\sum_{j=1}^{s_{i}}\overline{y}^{i}_{j}h^{i}_{j}(x)\right)\geq0.
\end{align}
Note that each $f_{i}$ is defined as $f_{i}(x) = \sup_{(y_{1}^{i},\ldots,y_{s_{i}}^{i})\in\Omega_{i}}\{h_{0}^{i}(x) +\sum^{s_{i}}_{j=1}y_{j}^{i}h_{j}^{i}(x)\}.$
Since each $\Omega_{i}$ is a compact set, there exist $\overline{y}^{i}\in\Omega_{i},$ $i=m+1,m+2,$ such that
\begin{equation*}
f_{i}(x)=h_{0}^{i}(x)+\sum_{j=1}^{s_{i}}\overline{y}^{i}_{j}h^{i}_{j}(x), \ i=m+1,m+2.
\end{equation*}
It follows from \eqref{thm1rel3} that
{\small
\begin{align*}
0\leq h_{0}^{m+1}(x)+\sum_{j=1}^{s_{m+1}}\overline{y}^{m+1}_{j}h^{m+1}_{j}(x)+\overline{\gamma} \left(h_{0}^{m+2}(x)+\sum_{j=1}^{s_{m+2}}\overline{y}^{m+2}_{j}h^{m+2}_{j}(x)\right)+\sum_{i=1}^m\overline{\lambda}_{i}\left(h_{0}^{i}(x)+\sum_{j=1}^{s_{i}}\overline{y}^{i}_{j}h^{i}_{j}(x)\right).
\end{align*}
Define that
\begin{align*}
\Phi(x):= h_{0}^{m+1}(x)+\sum_{j=1}^{s_{m+1}}\overline{y}^{m+1}_{j}h^{m+1}_{j}(x)+\overline{\gamma} \left(h_{0}^{m+2}(x)+\sum_{j=1}^{s_{m+2}}\overline{y}^{m+2}_{j}h^{m+2}_{j}(x)\right)
+\sum_{i=1}^m\overline{\lambda}_{i}\left(h_{0}^{i}(x)+\sum_{j=1}^{s_{i}}\overline{y}^{i}_{j}h^{i}_{j}(x)\right).
\end{align*}
}

Since $h_{0}^{i}+\sum_{j=1}^{s_{i}}\overline{y}^{i}_{j}h^{i}_{j},$ $i\in I,$ are SOS-convex, $\overline{\gamma}\geq0,$ and $\overline\lambda_{i}\geq0,$ $i=1,\ldots,m,$ $\Phi$ is an SOS-convex polynomial.
Moreover, since $\inf_{x\in\mathbb{R}^{n}}\Phi(x)=0,$ by Lemma~\ref{lemma3}, there exists $\overline{x}\in\mathbb{R}^{n}$ such that $\Phi(\overline{x})=0$ (and so, $\nabla \Phi(\overline{x})=0$).
It follows from Lemma~\ref{lemma4} that $\Phi$ is a sum of squares polynomial, i.e,
\begin{align}\label{thm1rel4}
h_{0}^{m+1}+\sum_{j=1}^{s_{m+1}}\overline{y}^{m+1}_{j}h^{m+1}_{j}+\overline{\gamma} \left(h_{0}^{m+2}+\sum_{j=1}^{s_{m+2}}\overline{y}^{m+2}_{j}h^{m+2}_{j}\right)+\sum_{i=1}^m\overline{\lambda}_{i}\left(h_{0}^{i}+\sum_{j=1}^{s_{i}}\overline{y}^{i}_{j}h^{i}_{j}\right) \in \Sigma^{2}_{2d}.
\end{align}
On the other hand, for each $i\in I,$ as $\overline{y}^{i}=(\overline{y}^{i}_{1},\ldots,\overline{y}^{i}_{s_{i}})\in\Omega_{i},$ there exists $\overline{z}^{i}:=(\overline{z}^{i}_{1},\ldots,\overline{z}^{i}_{p_{i}})\in \mathbb{R}^{p_{i}}$ such that
\begin{equation}\label{thm1rel5}
A_{0}^{i}+\sum_{j=1}^{s_{i}}\overline{y}_{j}^{i}A_{j}^{i}+\sum_{\ell=1}^{p_{i}}\overline{z}_{\ell}^{i}B_{\ell}^{i}\succeq0, \ i\in I.
\end{equation}
For each $i=1,\ldots,m,$ let $\lambda_{0}^{i}:=\overline{\lambda}_{i},$ $\lambda_{j}^{i}:=\overline{\lambda}_{i}\overline{y}_{j}^{i},$ and $z_{\ell}^{i}:=\overline{\lambda}_{i}\overline{z}_{\ell}^{i},$  $j=1,\ldots,s_{i},$ $\ell=1,\ldots,p_{i}.$
Let $\lambda_{0}^{m+1}:=1,$ $\lambda_{j}^{m+1}:=\overline{y}_{j}^{m+1},$ $j=1,\ldots,s_{m+1},$ and $z_{\ell}^{m+1}:=\overline{z}_{\ell}^{m+1},$ $\ell=1,\ldots,s_{m+1}.$
Let $\lambda_{0}^{m+2}:=\overline{\gamma},$ $\lambda_{j}^{m+2}:=\overline{\gamma}\,\overline{y}_{j}^{m+2},$ $j=1,\ldots,s_{m+2},$ and $z_{\ell}^{m+2}:=\overline{\gamma}\,\overline{z}_{\ell}^{m+2},$ $\ell=1,\ldots,s_{m+2}.$
Then, from \eqref{thm1rel4}, we have
\begin{align}\label{thm1rel6}
\sum_{i=1}^{m+2}\left(\lambda_{0}^{i}h_{0}^{i} +\sum^{s_{i}}_{j=1}\lambda_{j}^{i}h_{j}^{i}\right)\in\Sigma_{2d}^{2}.
\end{align}
Applying Proposition~\ref{proposition1} into \eqref{thm1rel6}, there exists $\overline{X}\in S^{s(d)}_{+}$ such that
\begin{equation*}
\sum_{i=1}^{m+2}\left(\lambda_{0}^{i}(h_{0}^{i})_{\alpha} +\sum^{s_{i}}_{j=1}\lambda_{j}^{i}(h_{j}^{i})_{\alpha}\right)=\langle B_{\alpha}, \overline{X}\rangle, \ \alpha\in\mathbb{N}^{n}_{2d}.
\end{equation*}
Note that $\overline{\gamma}$ and $\overline{\lambda}_{i},$ $i=1,\ldots,m,$ are nonnegative.
It follows from \eqref{thm1rel5} that we have
\begin{equation*}
\lambda_{0}^{i}A_{0}^{i} +\sum^{s_{i}}_{j=1}\lambda_{j}^{i}A_{j}^{i} +\sum^{p_{i}}_{\ell=1}z_{\ell}^{i}B_{\ell}^{i}\succeq0, \ i\in I.
\end{equation*}
It means that the tuple $(\lambda^{1},\ldots,\lambda^{m+2},z^{1},\ldots,z^{m+2},\overline{X})$ is a feasible solution of the problem~\eqref{SDD}, where for each $i\in I,$ $\lambda^{i}:=(\lambda^{i}_{0},\lambda^{i}_{1},\ldots,\lambda^{i}_{s_{i}})$ and $z^{i}:=(z^{i}_{1},\ldots,z^{i}_{p_{i}}).$
This fact leads to
\begin{equation*}
 \inf\eqref{FP}=\overline{\gamma}=\lambda_{0}^{m+2}\leq\max\eqref{SDD}.
\end{equation*}

Next, we claim that $\inf\eqref{FP}\geq\max\eqref{SDD}.$
Let $(\lambda^{1},\ldots,\lambda^{m+2},z^{1},\ldots,z^{m+2},X)$ be any feasible solution of the problem~\eqref{SDD}, where $\lambda^{i}:=(\lambda^{i}_{0},\lambda^{i}_{1},\ldots,\lambda^{i}_{s_{i}})\in\mathbb{R}_{+}\times\mathbb{R}^{s_{i}}$ and $z^{i}:=(z^{i}_{1},\ldots,z^{i}_{p_{i}})\in\mathbb{R}^{p_{i}},$ $i\in I.$
Then, by Proposition~\ref{proposition1}, we have
\begin{align}
&\sum_{i=1}^{m+2}\left(\lambda_{0}^{i}h_{0}^{i} +\sum^{s_{i}}_{j=1}\lambda_{j}^{i}h_{j}^{i}\right)\in\Sigma_{2d}^{2},\label{thm1rel1}\\
&\lambda_{0}^{i}A_{0}^{i} +\sum^{s_{i}}_{j=1}\lambda_{j}^{i}A_{j}^{i} +\sum^{p_{i}}_{\ell=1}z_{\ell}^{i}B_{\ell}^{i}\succeq0,\ i\in I.\notag
 \end{align}
For each $i\in I,$ pick $\overline{y}^{i}:=(\overline{y}^{i}_{1},\ldots,\overline{y}^{i}_{s_{i}})\in\Omega_{i}.$
Then, by the definition of $\Omega_{i},$ there exist $\overline{z}^{i}:=(\overline{z}^{i}_{1},\ldots,\overline{z}^{i}_{p_{i}})\in \mathbb{R}^{p_{i}},$ $i\in I,$ such that $A^{i}_{0}+\sum^{s_{i}}_{j=1}\overline y_{j}^{i}A_{j}^{i}+\sum^{p_{i}}_{\ell=1}\overline z_{\ell}^{i}B_{\ell}^{i}\succeq 0.$
Now, for each $i\in I$ and each $j=1,\ldots,s_{i},$ put
\begin{equation*}
\widetilde{y}_{j}^{i}:=
\left\{
  \begin{array}{ll}
    \frac{\lambda_{j}^{i}}{\lambda_{0}^{i}}, & \textrm{if } \lambda_{0}^{i}>0, \\
    \overline{y}_{j}^{i}, & \textrm{if } \lambda_{0}^{i}=0,
  \end{array}
\right.
\quad \textrm{ and } \quad
\widetilde{z}_{\ell}^{i}:=
\left\{
  \begin{array}{ll}
    \frac{z_{\ell}^{i}}{\lambda_{0}^{i}}, & \textrm{if } \lambda_{0}^{i}>0, \\
    \overline{z}_{\ell}^{i}, & \textrm{if } \lambda_{0}^{i}=0.
  \end{array}
\right.
\end{equation*}
Note that $\lambda_{0}^{m+1}=1.$
Then, for each $i\in I,$ we have
\begin{equation*}
A_{0}^{i}+\sum_{j=1}^{s_{i}}\widetilde{y}^{i}_{j}A_{j}^{i}+\sum_{\ell=1}^{p_{i}}\widetilde{z}^{i}_{j}B_{\ell}^{i}=
\left\{
  \begin{array}{ll}
\displaystyle \frac{1}{\lambda_{0}^{i}}\left(\lambda_{0}^{i}A_{0}^{i}+\sum_{j=1}^{s_{i}}\lambda_{j}^{i}A_{j}^{i}+\sum_{\ell=1}^{p_{i}}z_{\ell}^{i}B_{\ell}^{i}\right), & \textrm{if } \lambda_{0}^{i}>0, \\
\displaystyle A_{0}^{i}+\sum_{j=1}^{s_{i}}\overline{y}_{j}^{i}A_{j}^{i}+\sum_{\ell=1}^{p_{i}}\overline{z}_{\ell}^{i}B_{\ell}^{i}, & \textrm{if } \lambda_{0}^{i}=0,
  \end{array}
\right.
\end{equation*}
which is a positive semidefinite matrix.
So, we see that $\widetilde{y}^{i}:=(\widetilde{y}^{i}_{1},\ldots,\widetilde{y}^{i}_{s_{i}})\in\Omega_{i},$ $i\in I.$
Note that each set $\Omega_{i}$ is compact.
It follows from \cite[Lemma 4.1]{Chieu2018} that for each $i\in I,$ if $\lambda_{0}^{i}=0,$ then
$\lambda_{j}^{i}=0$ for all $j=1,\ldots,s_{i}.$
This implies that
\begin{equation*}
\lambda_{0}^{i}\left(h_{0}^{i}(x)+\sum^{s_{i}}_{j=1}\widetilde{y}_{j}^{i}h_{j}^{i}(x)\right)=\lambda_{0}^{i}h_{0}^{i}(x)+\sum^{s_{i}}_{j=1}\lambda_{0}^{i}\widetilde{y}_{j}^{i}h_{j}^{i}(x)=\lambda_{0}^{i}h_{0}^{i}(x)+\sum_{j=1}^{s_{i}}\lambda_{j}^{i}h_{j}^{i}(x).
\end{equation*}
This, together with \eqref{thm1rel1}, yields that for any $x\in K,$
\begin{align*}
&f_{m+1}(x)+\lambda_{0}^{m+2} f_{m+2}(x)\notag\\
\geq\ &\lambda_{0}^{m+1}\left(h_{0}^{m+1}(x)+\sum^{s_{m+1}}_{j=1}\widetilde{y}_{j}^{m+1}h_{j}^{m+1}(x)\right)+\lambda_{0}^{m+2}\left(h_{0}^{m+2}(x)+\sum^{s_{m+2}}_{j=1}\widetilde{y}_{j}^{m+2}h_{j}^{m+2}(x)\right)\\
\geq\ &\sum_{i=1}^{m+2}\left(\lambda_{0}^{i}h_{0}^{i}(x) +\sum^{s_{i}}_{j=1}\lambda_{j}^{i}h_{j}^{i}(x)\right)\geq0,
\end{align*}
i.e., $f_{m+1}(x)+\lambda_{0}^{m+2}f_{m+2}(x)\geq0$ for all $x\in K.$
So, we have
\begin{equation*}
 \inf\eqref{FP}\geq \max\eqref{SDD},
\end{equation*}
and thus, the proof is complete.
\end{proof}

\begin{remark}{\rm
As mentioned above, if the Slater condition holds for \eqref{FP}, then Assumption~\ref{assumption1} is also satisfied.
As a result, Theorem~\ref{thm1} still holds true under the Slater condition.
}\end{remark}

\subsection{Extracting optimal solutions}
As we have seen, the problem~\eqref{SDD} is essentially an SDP problem, which can be solved efficiently.
However, by solving the SDP problem~\eqref{SDD}, we can only obtain the optimal value of the problem~\eqref{FP} due to Theorem~\ref{thm1}.
Therefore, in order to get an optimal solution, one usually employs the Dinkelbach method (a parameter approach), by which one has to solve more SDP problems (see, e.g., \cite{Lee2018,Jiao2019b}).

In this part, we do not consider any parameter methods at all.
Instead, we propose a parameter-free approach for solving the problem~\eqref{FP}, and by which we can extract an optimal solution by solving one and only one SDP problem.
Note that there is no information on optimal solutions in the problem~\eqref{SDD}.
To proceed, we formulate the following {\it Lagrangian} dual problem of the problem~\eqref{SDD}.
\leqnomode
\begin{align}\label{SDP}
\inf\limits_{\substack{y\in\mathbb{R}^{s(2d)}\\Z_{i}\succeq0,i\in I}} \ \ & \sum_{\alpha\in\mathbb{N}^{n}_{2d}} \left(h_{0}^{m+1}\right)_{\alpha} y_{\alpha}+\left\langle A_{0}^{m+1},Z_{m+1}\right\rangle \tag{${\rm Q}$}\\
{\rm s.t.}\quad &  1+ \sum\limits_{\alpha\in\mathbb{N}_{2d}^n}\left(h_{0}^{m+2}\right)_{\alpha} y_{\alpha}+\left\langle A_{0}^{m+2},Z_{m+2}\right\rangle\leq0,\notag\\
&\sum_{\alpha\in\mathbb{N}_{2d}^n} \left(h_{0}^{i}\right)_{\alpha} y_{\alpha}+\left\langle A_{0}^{i},Z_{i} \right\rangle\leq0, \ i=1,\ldots,m,\notag\\
&\sum_{\alpha\in\mathbb{N}_{2d}^n} \left(h_{j}^{i}\right)_{\alpha} y_{\alpha}+\left\langle A_{j}^{i}, Z_{i}\right\rangle=0, \ i\in I, \ j=1,\ldots, s_{i},\notag\\
& \left\langle B_{\ell}^{i}, Z_{i} \right\rangle=0, \ i\in I, \ \ell=1,\ldots,p_{i},\notag\\
& \sum_{\alpha\in\mathbb{N}_{2d}^n}y_{\alpha} \mathcal{B}_{\alpha}\succeq0.\notag
\end{align}
It is worth noting that the problem \eqref{SDP} is also an SDP problem, which can be efficiently solved via interior point methods.
We now give the final result, which provides a way to extract an optimal solution of the problem~\eqref{FP} by solving the associated SDP problem~\eqref{SDP}, and the proof is motivated by~\cite[Theorem 4.2]{Chieu2018}.
Before that, we need the following common adopted assumption.
\begin{assumption}\label{assumption2}
For each $i\in I,$ there exist $\hat{y}^{i}\in\mathbb{R}^{s_{i}}$ and $\hat{z}^{i}\in\mathbb{R}^{p_{i}}$  such that $A_{0}^{i}+\sum_{j=1}^{s_{i}}\hat{y}_{j}^{i}A_{j}^{i}+\sum_{\ell=1}^{p_{i}}\hat{z}_{\ell}^{i}B_{\ell}^{i}\succ0.$
\end{assumption}

\begin{theorem}[\textbf{recovering optimal solutions from relaxations}]\label{thm2}
Suppose that Assumption~\ref{assumption1} and Assumption~\ref{assumption2} hold.
Then$,$ we have
\begin{equation*}
\inf\eqref{FP}=\sup\eqref{SDD} =\inf\eqref{SDP}.
\end{equation*}
Moreover$,$ if $(\overline{y}, \overline{Z}_{1},\ldots,\overline{Z}_{m+2})$ is an optimal solution of the problem~\eqref{SDP} with $\overline{y}_0\neq0,$
then $\overline{x}:=\frac{1}{\overline{y}_0}(L_{\overline{y}}(x_{1}),\ldots,L_{\overline{y}}(x_{n}))$ is an optimal solution of the problem~\eqref{FP}.
\end{theorem}
\begin{proof}
Note that $\inf\eqref{FP}=\sup\eqref{SDD}$ holds by Theorem~\ref{thm1} (under Assumption~\ref{assumption1}).
To proof the second equality in the first argument, we will show that $\inf\eqref{FP}\geq\inf\eqref{SDP}$ and $\inf\eqref{SDP}\geq\sup\eqref{SDD}.$

First, we prove that $\inf\eqref{FP}\geq\inf\eqref{SDP}.$
To see this, let $x$ be any feasible solution of the problem~\eqref{FP}, and let $\gamma:=\frac{f_{m+1}(x)}{-f_{m+2}(x)}.$
Then, we have $f_{m+1}(x)+\gamma f_{m+2}(x)=0$ and $f_i(x)\leq0,$ $i=1,\ldots,m.$
Recall that
\begin{equation*}
f_{i}(x) = \sup_{\left(y_{1}^{i},\ldots,y_{s_{i}}^{i}\right) \in \Omega_{i}}\left\{h_{0}^{i}(x) +\sum^{s_{i}}_{j=1}y_{j}^{i}h_{j}^{i}(x)\right\}, \ i\in I.
\end{equation*}
So, by the definition of each $f_i,$ we see that
\begin{align*}
&f_{m+1}(x)+{\gamma}\left(h_{0}^{m+2}(x) +\sum^{s_{m+2}}_{j=1}y_{j}^{m+2}h_{j}^{m+2}(x)\right)\leq 0,\\
&h_{0}^{m+1}(x) +\sum^{s_{m+1}}_{j=1}y_{j}^{m+1}h_{j}^{m+1}(x)+{\gamma}f_{m+2}(x)\leq 0,
\end{align*}
and
\begin{equation*}
h_{0}^{i}(x) +\sum^{s_{i}}_{j=1}y_{j}^{i}h_{j}^{i}(x)\leq f_{i}(x)\leq0, \ i = 1, \ldots, m,
\end{equation*}
for all $(y_{1}^{i},\ldots,y_{s_{i}}^{i})\in\Omega_{i},$ $i\in I,$ where each compact set $\Omega_{i}$ is given by
\begin{equation*}
 \Omega_{i}:= \left\{(y_{1}^{i},\ldots,y_{s_{i}}^{i})\in \mathbb{R}^{s_{i}} : \exists z^{i}\in \mathbb{R}^{p_{i}} \textrm{ s.t. } A_{0}^{i} +\sum^{s_{i}}_{j=1}y_{j}^{i}A_{j}^{i} +\sum^{p_{i}}_{\ell=1}z_{\ell}^{i}B_{\ell}^{i} \succeq 0\right\}.
 \end{equation*}
By Assumption~\ref{assumption2} and the strong duality theorem for semidefinite programming \cite{Vandenberghe1995}, there exist $W_{i}\succeq0,$ $i\in I,$ such that
\begin{equation}\label{thm2rel1}
\left\{
\begin{array}{l}
h_{0}^{m+1}(x)+\langle A_{0}^{m+1},W_{m+1}\rangle+{\gamma}f_{m+2}(x) \leq0,\\
f_{m+1}(x)+{\gamma}\left(h_0^{m+2}(x)+\langle A_{0}^{m+2},W_{m+2}\rangle\right) \leq0,\\
h_{0}^{i}(x)+\langle A_{0}^{i},W_{i}\rangle\leq0, \ i=1,\ldots,m,\\
h_{j}^{i}(x)+\langle A_{j}^{i},W_{i}\rangle=0, \  i\in I, \ j=1,\ldots,s_{i},\\
\langle B_{\ell}^{i},W_{i}\rangle=0, \  i\in I, \ \ell=1,\ldots,p_{i}.
\end{array}
\right.
\end{equation}
Since $x$ is a feasible solution of the problem~\eqref{FP}, hence $-f_{m+2}(x)>0.$
Dividing all inequalities and equalities in \eqref{thm2rel1} by $-f_{m+2}(x)$ and letting $Z_i:=\frac{W_i}{-f_{m+2}(x)},$ $i\in I,$ \eqref{thm2rel1} can be rewritten as
\begin{equation}\label{thm2rel1-1}
\left\{
\begin{array}{l}
\frac{h_{0}^{m+1}(x)}{-f_{m+2}(x)}+\langle A_{0}^{m+1},Z_{m+1}\rangle \leq\gamma,\\
1+\frac{h_0^{m+2}(x)}{-f_{m+2}(x)}+\langle A_{0}^{m+2},Z_{m+2}\rangle \leq0,\\
\frac{h_{0}^{i}(x)}{-f_{m+2}(x)}+\langle A_{0}^{i},Z_{i}\rangle\leq0, \ i=1,\ldots,m,\\
\frac{h_{j}^{i}(x)}{-f_{m+2}(x)}+\langle A_{j}^{i},Z_{i}\rangle=0, \  i\in I, \ j=1,\ldots,s_{i},\\
\langle B_{\ell}^{i},Z_{i}\rangle=0, \  i\in I, \ \ell=1,\ldots,p_{i}.
\end{array}
\right.
\end{equation}
Let $y:=\frac{v_{d}(x)}{-f_{m+2}(x)}\in\mathbb{R}^{s(2d)}.$
Then, \eqref{thm2rel1-1} can be rewritten as follows:
\begin{equation*}
\left\{
\begin{array}{l}
\sum\limits_{\alpha\in\mathbb{N}_{2d}^n}(h_{0}^{m+1})_{\alpha} y_{\alpha}+\langle A_{0}^{m+1},Z_{m+1}\rangle \leq\gamma,\\
1+\sum\limits_{\alpha\in\mathbb{N}_{2d}^n}(h_{0}^{m+2})_{\alpha} y_{\alpha}+\langle A_{0}^{m+2},Z_{m+2}\rangle \leq0,\\
\sum\limits_{\alpha\in\mathbb{N}_{2d}^n}(h_{0}^{i})_{\alpha} y_{\alpha}+\langle A_{0}^{i},Z_{i}\rangle\leq0, \ i=1,\ldots,m,\\
\sum\limits_{\alpha\in\mathbb{N}_{2d}^n}(h_{j}^{i})_{\alpha} y_{\alpha}+\langle A_{j}^{i},Z_{i}\rangle=0, \  i\in I, \ j=1,\ldots,s_{i},\\
\langle B_{\ell}^{i},Z_{i}\rangle=0, \  i\in I, \ \ell=1,\ldots,p_{i}.
\end{array}
\right.
\end{equation*}
In addition, since $yy^{T} = \sum_{\alpha\in\mathbb{N}_{2d}^n}y_{\alpha} B_{\alpha}\succeq0,$
$(y, Z_{1},\ldots, Z_{m+2})$ is a feasible solution of the problem~\eqref{SDP}, and hence,
\begin{equation*}
\frac{f_{m+1}(x)}{-f_{m+2}(x)}=\gamma\geq \sum\limits_{\alpha\in\mathbb{N}_{2d}^n}(h_{0}^{m+1})_{\alpha} y_{\alpha}+\langle A_{0}^{m+1},Z_{m+1}\rangle\geq \inf\eqref{SDP}.
\end{equation*}
Since $x\in K$ is arbitrary, we have $\inf\eqref{FP}\geq \inf\eqref{SDP}.$

To complete the proof of the first argument, we now show that $\inf\eqref{SDP}\geq \sup\eqref{SDD}.$
Let $y\in\R^{s(2d)}$ and $Z_i\succeq0,$ $i\in I,$ be any feasible solution of the problem~\eqref{SDP},
and let $\lambda_{j}^{i}\in \R,$ $j=1,\ldots, s_i,$ $z_{\ell}^{i}\in\R,$ $\ell=1,\ldots,p_i,$ $i\in I,$ $X\succeq0$ be any feasible solution of the problem~\eqref{SDD}.
Then, we have
\begin{equation*}
\sum_{i=1}^{m+2}\left(\lambda_{0}^{i}\left(h_{0}^{i}\right)_{\alpha} +\sum^{s_{i}}_{j=1}\lambda_{j}^{i}\left(h_{j}^{i}\right)_{\alpha}\right)= \left\langle \mathcal{B}_{\alpha}, X \right\rangle, \ \alpha\in\mathbb{N}^{n}_{2d},
\end{equation*}
Multiplying it by $y_\alpha,$ $\alpha\in \N^n_{2d}$ and summing them, we have
{\small
\begin{align*}
0\leq \ & \left\langle \sum_{\alpha\in\mathbb{N}_{2d}^n}y_{\alpha} \mathcal{B}_{\alpha}, X\right\rangle\\
=\ &\sum_{\alpha\in\mathbb{N}_{2d}^n}\left(h_{0}^{m+1}\right)_{\alpha}y_{\alpha}+\lambda_{0}^{m+2}\sum_{\alpha\in\mathbb{N}_{2d}^n}\left(h_{0}^{m+2}\right)_{\alpha}y_{\alpha}+\sum_{i=1}^{m}\lambda_{0}^{i}\sum\limits_{\alpha\in\mathbb{N}_{2d}^n}\left(h_{0}^{i}\right)_{\alpha}y_{\alpha} +\sum_{i=1}^{m+2}\sum^{s_{i}}_{j=1}\lambda_{j}^{i}\sum\limits_{\alpha\in\mathbb{N}_{2d}^n}\left(h_{j}^{i}\right)_{\alpha}y_{\alpha}\\
\leq \ & \sum_{\alpha\in\mathbb{N}_{2d}^n}\left(h_{0}^{m+1}\right)_{\alpha}y_{\alpha}+\lambda_{0}^{m+2}\left(-1-\left\langle A_{0}^{m+2}, Z_{m+2}\right\rangle\right)-\sum_{i=1}^{m}\lambda_{0}^{i}\left\langle A_{0}^{i}, Z_{i}\right\rangle-\sum_{i=1}^{m+2}\sum^{s_{i}}_{j=1}\lambda_{j}^{i}\left\langle A_{j}^{i}, Z_{i}\right\rangle\\
= \ & \sum_{\alpha\in\mathbb{N}_{2d}^n}\left(h_{0}^{m+1}\right)_{\alpha}y_{\alpha}-\lambda_{0}^{m+2}-\left\langle \lambda_{0}^{m+2}A_{0}^{m+2}+\sum^{s_{m+2}}_{j=1}\lambda_{j}^{m+2} A_{j}^{m+2}, Z_{m+2}\right\rangle-\left\langle \sum^{s_{m+1}}_{j=1}\lambda_{j}^{m+1} A_{j}^{m+1}, Z_{m+2}\right\rangle\\
&\ -\sum_{i=1}^{m}\left\langle \lambda_{0}^{i}A_{0}^{i}+\sum^{s_{i}}_{j=1}\lambda_{j}^{i}A_{j}^{i}, Z_{i}\right\rangle\\
\leq \ & \sum_{\alpha\in\mathbb{N}_{2d}^n}\left(h_{0}^{m+1}\right)_{\alpha}y_{\alpha}-\lambda_{0}^{m+2}+\left\langle \sum^{p_{m+2}}_{\ell=1}z_{\ell}^{m+2}B_{\ell}^{m+2}, Z_{m+2}\right\rangle+\left\langle A_{0}^{m+1}+\sum^{p_{m+1}}_{\ell=1}z_{\ell}^{m+1}B_{\ell}^{m+1}, Z_{m+1}\right\rangle\\
&\ + \sum_{i=1}^{m}\left\langle \sum^{p_{i}}_{\ell=1}z_{\ell}^{i}B_{\ell}^{i}, Z_{i}\right\rangle\\
= & \sum_{\alpha\in\mathbb{N}_{2d}^n}\left(h_{0}^{m+1}\right)_{\alpha}y_{\alpha}+\left\langle A_{0}^{m+1}, Z_{m+1}\right\rangle-\lambda_{0}^{m+2},
\end{align*}}
i.e., $\sum\limits_{\alpha\in\mathbb{N}_{2d}^n}\left(h_{0}^{m+1}\right)_{\alpha}y_{\alpha}+\left\langle A_{0}^{m+1}, Z_{m+1}\right\rangle\geq \lambda_{0}^{m+2}.$
So, we have $\inf\eqref{SDP}\geq \sup\eqref{SDD}.$

To prove the second argument, let $(\overline{y}, \overline{Z}_{1},\ldots,\overline{Z}_{m+2})$ be an optimal solution of the problem~\eqref{SDP} and assume that $\overline{y}_0\neq0.$
Then, we have
\begin{align*}
&1+ \sum\limits_{\alpha\in\mathbb{N}_{2d}^n}\left(h_{0}^{m+2}\right)_{\alpha} \overline{y}_{\alpha}+\left\langle A_{0}^{m+2},\overline{Z}_{m+2}\right\rangle\leq0,\\
&\sum\limits_{\alpha\in\mathbb{N}_{2d}^n}(h_{0}^{i})_{\alpha} \overline{y}_{\alpha}+\langle A_{0}^{i},\overline{Z}_{i}\rangle\leq0, \ i=1,\ldots,m,\\
&\sum\limits_{\alpha\in\mathbb{N}_{2d}^n}(h_{j}^{i})_{\alpha} \overline{y}_{\alpha}+\langle A_{j}^{i},\overline{Z}_{i}\rangle=0, \  i\in I, \ j=1,\ldots,s_{i},\\
&\langle B_{\ell}^{i},\overline{Z}_{i}\rangle=0, \  i\in I, \ \ell=1,\ldots,p_{i},\\
&\sum_{\alpha\in\mathbb{N}_{2d}^n}\overline{y}_{\alpha} \mathcal{B}_{\alpha}\succeq0.
\end{align*}
Putting $\overline{z}_{\alpha}:=\frac{\overline{y}_{\alpha}}{\overline{y}_{0}},$ $\alpha\in\N^{n}_{2d},$ and $\overline{W}_i:=\frac{\overline{Z}_{i}}{\overline{y}_{0}},$ we have
\begin{align}
&\frac{1}{\overline{y}_{0}}+ \sum\limits_{\alpha\in\mathbb{N}_{2d}^n}\left(h_{0}^{m+2}\right)_{\alpha} \overline{z}_{\alpha}+\left\langle A_{0}^{m+2},\overline{W}_{m+2}\right\rangle\leq0,\label{thm2rel9}\\
&\sum\limits_{\alpha\in\mathbb{N}_{2d}^n}(h_{0}^{i})_{\alpha} \overline{z}_{\alpha}+\langle A_{0}^{i},\overline{W}_{i}\rangle\leq0, \ i=1,\ldots,m,\label{thm2rel2}\\
&\sum\limits_{\alpha\in\mathbb{N}_{2d}^n}(h_{j}^{i})_{\alpha} \overline{z}_{\alpha}+\langle A_{j}^{i},\overline{W}_{i}\rangle=0, \  i\in I, \ j=1,\ldots,s_{i},\label{thm2rel3}\\
&\langle B_{\ell}^{i},\overline{W}_{i}\rangle=0, \  i\in I, \ \ell=1,\ldots,p_{i},\label{thm2rel4}\\
&\sum_{\alpha\in\mathbb{N}_{2d}^n}\overline{z}_{\alpha} \mathcal{B}_{\alpha}\succeq0, \  \overline{z}_{0}=1.\label{thm2rel5}
\end{align}
Let $(y_{1}^{i},\ldots,y_{s_{i}}^{i})\in\Omega_{i},$ $i\in I,$ be any given.
Then, for each $i\in I,$ there exists $(z_{1}^{i},\ldots,z_{p_{i}}^{i})\in\mathbb{R}^{p_{i}}$ such that
\begin{equation*}
A_{0}^{i} +\sum^{s_{i}}_{j=1}y_{j}^{i}A_{j}^{i} +\sum^{p_{i}}_{\ell=1}z_{\ell}^{i}B_{\ell}^{i} \succeq 0.
\end{equation*}
It follows from (\ref{thm2rel2}) that for each $i=1,\ldots,m,$
\begin{align}
0\geq\sum\limits_{\alpha\in\mathbb{N}_{2d}^n}(h_{0}^{i})_{\alpha} \overline{z}_{\alpha}+\langle A_{0}^{i},\overline{W}_{i}\rangle
\geq\ &\sum\limits_{\alpha\in\mathbb{N}_{2d}^n}(h_{0}^{i})_{\alpha} \overline{z}_{\alpha}-\sum^{s_{i}}_{j=1}y_{j}^{i}\langle A_{j}^{i},\overline{W}_{i}\rangle -\sum^{p_{i}}_{\ell=1}z_{\ell}^{i}\langle B_{\ell}^{i},\overline{W}_{i}\rangle\notag\\
= \ &\sum\limits_{\alpha\in\mathbb{N}_{2d}^n}(h_{0}^{i})_{\alpha} \overline{z}_{\alpha}+\sum^{s_{i}}_{j=1}y_{j}^{i}\sum\limits_{\alpha\in\mathbb{N}_{2d}^n}(h_{j}^{i})_{\alpha} \overline{z}_{\alpha}\notag\\
= \ &L_{\overline{z}}\left(h_{0}^{i}+\sum_{j=1}^{s_{i}}y_{j}^{i}h_{j}^{i}\right),\label{thm2rel6}
\end{align}
where the first equality follows from (\ref{thm2rel3}) and (\ref{thm2rel4}).
Note that for each $i=1,\ldots,m,$ $h_{0}^{i}+\sum_{j=1}^{s_{i}}y_{j}^{i}h_{j}^{i}$ is SOS-convex.
Since $\overline{z}$ satisfies (\ref{thm2rel5}), by Lemma~\ref{lemma5}, we see that for each $i=1,\ldots,m,$
\begin{equation*}
L_{\overline{z}}\left(h_{0}^{i}+\sum_{j=1}^{s_{i}}y_{j}^{i}h_{j}^{i}\right)\geq \left(h_{0}^{i}+\sum_{j=1}^{s_{i}}y_{j}^{i}h_{j}^{i}\right)(L_{\overline{z}}(x_{1}),\ldots,L_{\overline{z}}(x_{n}))=h_{0}^{i}(\overline{x})+\sum_{j=1}^{s_{i}}y_{j}^{i}h_{j}^{i}(\overline{x}).
\end{equation*}
This, together with (\ref{thm2rel6}), yields that for each $i=1,\ldots,m,$
\begin{equation}\label{thm2rel7}
0\geq \sup_{(y_{1}^{i},\ldots,y_{s_{i}}^{i})\in\Omega_{i}}\left\{h_{0}^{i}(\overline{x})+\sum_{j=1}^{s_{i}}y_{j}^{i}h_{j}^{i}(\overline{x})\right\}=f_{i}(\overline{x})
\end{equation}
as each $(y_{1}^{i},\ldots,y_{s_{i}}^{i})\in\Omega_{i}$ is arbitrary.
So, $\overline{x}$ is a feasible solution of the problem~\eqref{FP}.
Moreover, by similar arguments as above into \eqref{thm2rel9}, we have
\begin{equation*}
\frac{1}{\overline{y}_{0}}\leq -\sum\limits_{\alpha\in\mathbb{N}_{2d}^n}\left(h_{0}^{m+2}\right)_{\alpha} \overline{z}_{\alpha}-\left\langle A_{0}^{m+2},\overline{W}_{m+2}\right\rangle\leq -L_{\overline{z}}\left(h_{0}^{m+2}+\sum_{j=1}^{s_{m+2}}y_{j}^{m+2}h_{j}^{m+2}\right)\leq-f_{m+2}(\overline{x}),
\end{equation*}
i.e., $\overline{y}_0\geq \frac{-1}{f_{m+2}(\overline{x})}.$
This implies that
\begin{align*}
\inf\eqref{SDP}=\sum_{\alpha\in\mathbb{N}^{n}_{2d}} \left(h_{0}^{m+1}\right)_{\alpha} \overline{y}_{\alpha}+\left\langle A_{0}^{m+1},\overline{Z}_{m+1}\right\rangle
=& \ \overline{y}_0\left(\sum_{\alpha\in\mathbb{N}^{n}_{2d}} \left(h_{0}^{m+1}\right)_{\alpha} \overline{z}_{\alpha}+\left\langle A_{0}^{m+1},\overline{W}_{m+1}\right\rangle\right)\\
\geq& \ \overline{y}_0L_{\overline{z}}\left(h_{0}^{m+1}+\sum_{j=1}^{s_{m+1}}y_{j}^{m+1}h_{j}^{m+1}\right)\\
\geq& \ \overline{y}_0f_{m+1}(\overline{x})\geq \frac{-f_{m+1}(\overline{x})}{f_{m+2}(\overline{x})}\\
\geq& \ \inf\eqref{FP}=\inf\eqref{SDP}.
\end{align*}
Thus, $\overline{x}:=\frac{1}{\overline{y}_0}(L_{\overline{y}}(x_{1}),\ldots,L_{\overline{y}}(x_{n}))$ is an optimal solution of the problem~\eqref{FP}.
\end{proof}


\subsection{Illustrative examples}
In this subsection, we give two examples to illustrate our findings.
The first example has been calculated in \cite{Jiao2019b} by using the Dinkelbach method.
Below, we re-examine it by using Theorem~\ref{thm2}.
\begin{example}[{see \cite[Example 3.6]{Jiao2019b}}]{\rm
Consider the following fractional program,
\begin{align}\label{FP0}
\min\limits_{x\in\mathbb{R}^{2}}\ \left\{\frac{f_{2}(x)}{-f_{3}(x)} \colon f_{1}(x)\leq0\right\}, \tag{${{\rm FP}_{1}}$}
\end{align}
where
\begin{align*}
f_1(x_1, x_2) \ & = x_1^2+x_1x_2+x_2^2+4x_1+1 + s((x_1,x_2)\mid \Gamma_1), \\
f_2(x_1, x_2) \ & = x_1^8+x_1^2+x_1x_2+x_2^2 + s((x_1,x_2)\mid \Gamma_2), \\
f_3(x_1, x_2) \ & = (x_1+x_2)^2 - 10 + s((x_1,x_2)\mid \Gamma_3),
\end{align*}
here, $\Gamma_i=[-1,1]^2,$ $i=1,2,3,$ and $s((x_1,x_2)\mid \Gamma_i)$ stands for support functions, defined as
\begin{equation*}
s((x_1,x_2)\mid \Gamma_i)=\max_{u_i\in \Gamma_i}\langle u_i,x\rangle=|x_1|+|x_2|, \ i=1,2,3,
\end{equation*}
Note that the functions $f_i,$ $i=1,2,3,$ are SOS-convex semi-algebrac, which can be easily checked.
Let $K_1:=\{(x_1,x_2) : f_1(x_1,x_2)\leq 0\}$ be the feasible set of the problem~\eqref{FP0}.
By \cite[Example 3.6]{Jiao2019b}, we know the optimal solution and optimal value are $(\overline x_1, \overline x_2) = (0.3820, 0) \approx (\frac{-3+\sqrt{5}}{2}, 0)$ and $0.0558,$ respectively.

In order to match the notation used in the function~\eqref{function f_{i}}.
We let
$h_{0}^{1}(x)= x_1^2+x_1x_2+x_2^2+4x_1+1,$ $h_{j}^{1}(x)=x_{j},$ $j=1,2,$ $h_{0}^{2}(x)=x_1^8+x_1^2+x_1x_2+x_2^2,$ $h_{j}^{2}(x)=x_{j},$ $j=1,2,$ $h_{0}^{3}(x)=(x_1+x_2)^2 - 10,$ $h_{j}^{3}(x)=x_{j},$ $j=1,2,$  and each $\Omega_{j}$ is given by
\begin{align*}
	\Omega_{j}= \ &\left\{(y_{1}^{j},y_{2}^{j})\in\mathbb{R}^{2} : I_{4}+y_{1}^{j}\left(\begin{array}{rrrr}  -1 & 0 & 0 & 0 \\  0 & 1 & 0 & 0\\ 0 & 0 & 0 & 0\\ 0 & 0 & 0 & 0\\ \end{array}\right)+y_{2}^{j}\left(\begin{array}{rrrr}  0 & 0 & 0 & 0 \\  0 & 0 & 0 & 0\\ 0 & 0 & -1 & 0\\ 0 & 0 & 0 & 1\\ \end{array}\right)\succeq0\right\},\ j=1,2,3,
\end{align*}
where $I_4$ denotes the identity matrix of order $4.$
By simple calculations, we see that $\Omega_{j}=\{(y_{1}^{j},y_{2}^{j})\in\mathbb{R}^{2} : -1\leq y_{1}^{j}\leq1, \ -1\leq y_{2}^{j}\leq1\},j=1,2,3. $
By the results, we also see that
$f_{1}(x)= x_1^2+x_1x_2+x_2^2+4x_1+|x_{1}|+|x_{2}|+1,$ $f_{2}(x)=x_1^8+x_1^2+x_1x_2+x_2^2+|x_{1}|+|x_{2}|,$ and $f_{3}(x)=(x_1+x_2)^2+|x_{1}|+|x_{2}|-10.$

Now, we consider the following sum of squares relaxation dual problem for the problem~\eqref{FP1}:
\begin{align*}
	\sup\limits_{\substack{\lambda_{0}^{i}, \lambda_{j}^{i}}} \ \ & \lambda_{0}^{3}  \\
	\textrm{s.t.}\quad &\lambda_{0}^{1}h_{0}^{1}+\sum^{2}_{j=1}\lambda_{j}^{1}h_{j}^{1}+\lambda_{0}^{2}h_{0}^{2}+\sum^{2}_{j=1}\lambda_{j}^{2}h_{j}^{2}+\lambda_{0}^{3}h_{0}^{3}+\sum^{2}_{j=1}\lambda_{j}^{3}h_{j}^{3}\in\Sigma^{2}_{8},\\
	&{\rm diag}\left( \lambda_{0}^{1}-\lambda_{1}^{1},\  \lambda_{0}^{1}+\lambda_{1}^{1}, \  \lambda_{0}^{1}-\lambda_{2}^{1},\ \lambda_{0}^{1}+\lambda_{2}^{1}\right)\succeq0,\\
	&{\rm diag}\left( \lambda_{0}^{2}-\lambda_{1}^{2},\  \lambda_{0}^{2}+\lambda_{1}^{2},\  \lambda_{0}^{2}-\lambda_{2}^{2},\ \lambda_{0}^{2}+\lambda_{2}^{2} \right)\succeq0,\\
	&{\rm diag}\left( \lambda_{0}^{3}-\lambda_{1}^{3}, \  \lambda_{0}^{3}+\lambda_{1}^{3}, \  \lambda_{0}^{3}-\lambda_{2}^{3},\ \lambda_{0}^{3}+\lambda_{2}^{3}\right)\succeq0,\\
	&\lambda_{0}^{2}=1, \ \lambda_{0}^{i}\geq0,\ i=1,3, \ \lambda_{j}^{1}\in\mathbb{R}, \ \lambda_{j}^{2}\in\mathbb{R}, \ \lambda_{j}^{3}\in\mathbb{R}, \  j=1,2,
\end{align*}
where ${\rm diag}(w)$ stands for a square diagonal matrix with the elements of vector $w$ on the diagonal.
Invoking Proposition~\ref{proposition1}, there exists $X\in S_{+}^{s(4)}(=S_{+}^{15})$ such that
\begin{align}
	&\lambda_{0}^{1}h_{0}^{1}(x)+\sum^{2}_{j=1}\lambda_{j}^{1}h_{j}^{1}(x)+\lambda_{0}^{2} h_{0}^{2}(x)+\sum^{2}_{j=1}\lambda_{j}^{2}h_{j}^{2}(x)+\lambda_{0}^{3}h_{0}^{3}(x)+\sum^{2}_{j=1}\lambda_{j}^{3}h_{j}^{3}(x) \label{ex0rel1}\\
	=\ &\langle v_{4}(x)v_{4}(x)^{T},X\rangle, \  \forall x\in\mathbb{R}^{2}.\notag
\end{align}
Thanks to \cite[Theorem 1]{Reznick1978}, we can reduce the size of $v_{4}(x),$ that is, 6, and so $X \in S^6_{+}.$
In more detail, $v_{4}(x)=(1,x_{1},x_{2},x_{1}^{2},x_{1}^{3},x_{1}^{4})^{T}$ in $(\ref{ex0rel1}).$
With this fact we formulate the following SDP dual problem for the problem~\eqref{FP0}:
\begin{align}\label{SDD0}
	\sup\limits_{\substack{\lambda_{j}^{i},X}} \ \ & \lambda_{0}^{3} \tag{$\widehat{\rm Q}_{1}$}\\
	\textrm{s.t.}\ \ \,   & \lambda^{1}_{0}-10\lambda^{3}_{0}=X_{11}, \,\lambda^{1}_{1}+\lambda^{2}_{1}+\lambda^{3}_{1}+4\lambda^{1}_{0}=2X_{12}, \, \lambda_{2}^{1}+\lambda^{2}_{2}+\lambda^{3}_{2}=2X_{13},\notag \\
	&\lambda^{1}_{0}+\lambda^{2}_{0}+\lambda^{3}_{0}=2X_{14}+X_{22}, \,\lambda^{1}_{0}+\lambda^{2}_{0}+2\lambda^{3}_{0}=2X_{23}, \,\lambda^{1}_{0}+\lambda^{2}_{0}+\lambda^{3}_{0}=X_{33}, \notag \\
	& 0=2X_{16}+2X_{25}+X_{44}, \,\lambda^{2}_{0}=X_{66}, \,0=X_{15}+X_{24}=X_{34}=X_{35},\notag \\
	&0=X_{26}+X_{45}=X_{36}=2X_{46}+X_{55}=X_{56}, \notag \\
		&{\rm diag}\left(\lambda_{0}^{1}-\lambda_{1}^{1},\  \lambda_{0}^{1}+\lambda_{1}^{1},\  \lambda_{0}^{1}-\lambda_{2}^{1},\ \lambda_{0}^{1}+\lambda_{2}^{1}\right)\succeq0,\notag\\
	&{\rm diag}\left(\lambda_{0}^{2}-\lambda_{1}^{2},\  \lambda_{0}^{2}+\lambda_{1}^{2},\  \lambda_{0}^{2}-\lambda_{2}^{2},\ \lambda_{0}^{2}+\lambda_{2}^{2}\right)\succeq0,\notag\\
	&{\rm diag}\left(\lambda_{0}^{3}-\lambda_{1}^{3},\  \lambda_{0}^{3}+\lambda_{1}^{3},\  \lambda_{0}^{3}-\lambda_{2}^{3},\ \lambda_{0}^{3}+\lambda_{2}^{3}\right)\succeq0,\notag\\
	&\lambda_{0}^{2}=1, \ \lambda_{0}^{i}\geq0,\ i=1,3, \ \lambda_{j}^{1}\in\mathbb{R}, \ \lambda_{j}^{2}\in\mathbb{R}, \ \lambda_{j}^{3}\in\mathbb{R}, \  j=1,2.\notag
\end{align}
To proceed, we formulate the following Lagrangian dual problem of the problem~\eqref{SDD0}:
\leqnomode
\begin{align}\label{SDP0}
	\inf\limits_{\substack{y\in\mathbb{R}^{15}\\Z_{i}\succeq0}} \ \ & \sum_{\alpha\in\mathbb{N}^{n}_{2d}} y_{80}+y_{20}+y_{11}+y_{02}+\langle A_{0}^{2},Z_{2}\rangle\tag{${\rm Q}_{2}$}\\
	{\rm s.t.}\ \ \ &1+y_{20}+y_{02}+2y_{11}-10y_{00}+\langle A_{0}^{3},Z_{3}\rangle\leq0,\notag\\
	&y_{20}+y_{11}+y_{02}+4y_{10}+y_{00}+\langle A_{0}^{1},Z_{1}\rangle\leq0,\notag\\
	&y_{10}+\langle A_{1}^{1},Z_{1}\rangle=0, \ y_{01}+\langle A_{2}^{1},Z_{1}\rangle=0,\notag\\
	&y_{10}+\langle A_{1}^{2},Z_{2}\rangle=0, \ y_{01}+\langle A_{2}^{2},Z_{2}\rangle=0,\notag\\
	&y_{10}+\langle A_{1}^{3},Z_{3}\rangle=0, \ y_{01}+\langle A_{2}^{3},Z_{3}\rangle=0,\notag\\
	& \sum_{\alpha}y_{\alpha} \mathcal{B}_{\alpha}\succeq0.\notag
\end{align}
\reqnomode
Solving the problem \eqref{SDP0} using CVX \cite{Grant2013} in MATLAB, we obtain the optimal value $0.0558$ and an optimal solution $(\overline y,\overline{Z}_{1},\overline{Z}_{2}, \overline{Z}_{3})$ for the problem~\eqref{SDP0}, where
\begin{align*}
	&\begin{array}{rrr}
		\overline{y}=(0.1056, -0.0403, 0.0000, 0.0154, -0.0000, 0.0000, -0.0059, 0.0000, \\
            0.0022, -0.0000, 0.0009, 0.0000, 0.0003, -0.0001, 0.0000 )^{T},
	\end{array}\\
	&\overline{Z}_{1}\approx{\rm diag}(0.0000, 0.0403, 0.0000, 0.0000),\\
	&\overline{Z}_{2}\approx{\rm diag}(0.0000, 0.0403, 0.0000, 0.0000),\\
	&\overline{Z}_{3}\approx{\rm diag}(0.0000, 0.0403, 0.0000, 0.0000).
\end{align*}
Thus, it follows from Theorem~\ref{thm2} that $\overline{x}=(-0.3820,0)$ is an optimal solution of the problem~\eqref{FP1}.
In a result, we see that the obtained results by solving one SDP problem due to
Theorem~\ref{thm2} are same to the results in \cite{Jiao2019b}.
\qed
}\end{example}

Below, we give another example to show how to find an optimal solution of the problem~\eqref{FP} by solving one SDP problem due to Theorem~\ref{thm2}.

\begin{example}{\rm
Consider the following fractional program,
\begin{align}\label{final}
\min\limits_{x\in\mathbb{R}^{2}}\left\{ \frac{\sqrt{x^TQx}}{\underset{\left (a_1, a_2\right)\,\in\, \Gamma}{\inf} \left\{a_1x_1+a_2x_2-b\right\} } \colon  f_{1}(x)\leq0\right\}, 
\end{align}
where $Q$ is a positive definite matrix given by
\begin{align*}
Q=\left(
    \begin{array}{cc}
      10 & 4 \\
      4 & 2 \\
    \end{array}
  \right),
\end{align*}
the set $\Gamma$ is given by
\begin{align*}
\Gamma &= \left\{(a_{1},a_{2})\in\mathbb{R}^{2} \colon \left(\begin{array}{rrrr}  1 & 0 & 0 & 0 \\  0 & 1 & 0 & 0 \\ 0 & 0 & 1 & 0 \\ 0 & 0 & 0 & 1 \\  \end{array}\right)+a_{1}\left(\begin{array}{rrrr}  -1 & 0 & 0 & 0  \\  0 & 1 & 0 & 0 \\ 0 & 0 & 0 & 0 \\ 0 & 0 & 0 & 0 \\  \end{array}\right)
+a_2\left(\begin{array}{rrrr}  0 & 0 & 0 & 0 \\  0 & 0 & 0 & 0 \\ 0 & 0 & -1 & 0 \\ 0 & 0 & 0 & 1 \\ \end{array}\right) \succeq0\right\},
\end{align*}
$b=a_1+a_2-3,$ and $f_{1} =  (x_{1}-2)^{2}+ (x_{2}-1)^{2}-1.$
Observe that the problem~\eqref{final} can be equivalently reformulated as follows,
\begin{align}\label{FP1}
\min\limits_{x\in\mathbb{R}^{2}}\left\{ \frac{\left \| Q^\frac{1}{2}x  \right \|_2 }{-\underset{\left ( a_1,a_2 \right )\, \in\, \Gamma}{\sup} \left\{-a_1x_1-a_2x_2+b\right\}} \colon f_{1}(x)\leq0 \right\}, \tag{${{\rm FP}_{2}}$}
\end{align}
where $Q^\frac{1}{2} := \begin{pmatrix}
					3 & 1\\
					1 & 1
				\end{pmatrix}.$
In order to match the notation used in the problem~\eqref{FP1}, we denote by
\begin{align*}
f_{1}(x)= \ &  (x_{1}-2)^{2}+ (x_{2}-1)^{2}-1 = h_{0}^{1}(x),  \notag\\
f_{2}(x)=\ & \left \| Q^\frac{1}{2}x  \right \|_2 = \sup_{(y_{1}^{2},y_{2}^{2})\in\Omega_2}\bigg\{h_{0}^{2}(x)+y_{1}^{2}h_{1}^{2}(x)+y_{2}^{2}h_{2}^{2}(x)\bigg\}, \notag\\
f_{3}(x)=\ & \underset{\left(a_1,a_2 \right)\, \in\, \Gamma}{\sup}\left\{-a_1x_1-a_2x_2+a_1+a_2-3 \right\} = \sup_{(a_1,a_2)\,\in\, \Gamma}\bigg\{h_{0}^{3}(x)+a_1h_{1}^{3}(x)+a_2h_{2}^{3}(x)\bigg\}, \notag
\end{align*}
where
$h_{0}^{1}(x)=(x_{1}-2)^{2}+ (x_{2}-1)^{2}-1,$  $h_{0}^{2}(x)=0,$ $h_{1}^{2}(x)=3x_{1}+x_{2},$
$h_{2}^{2}(x)=x_{1}+x_{2},$ $h_{0}^{3}(x)=-3,$ $h_{1}^{3}(x)=1-x_{1},$ $h_{2}^{3}(x)=1-x_{2},$  and $\Omega_2$ is given by
\begin{align*}
\Omega_{2}= \ &\left\{(y_{1}^{2},y_{2}^{2})\in\mathbb{R}^{2} \colon I_{3}+y_{1}^{2}\left(\begin{array}{rrr}   0 & 0 & 1 \\  0 & 0 & 0\\  1 & 0 & 0\\ \end{array}\right)+y_{2}^{2}\left(\begin{array}{rrr} 0 & 0 & 0 \\  0 & 0 & 1\\ 0 & 1 & 0\\ \end{array}\right)\succeq0\right\},
\end{align*}
where $I_{3}$ denotes the identity matrix of order $3.$
Then, we easily see that for each $(y_{1}^{2},y_{2}^{2})\in \Omega_2,$ each $(a_1,a_2)\in \Gamma,$ the functions $h_{0}^{2}(x)+y_{1}^{2}h_{1}^{2}(x)+y_{2}^{2}h_{2}^{2}(x)$ and $h_{0}^{3}(x)+a_{1}h_{1}^{3}(x)+a_{2}h_{2}^{3}(x)$ are SOS-convex polynomials, and so, for each $i=1,2,3,$ $f_{i}$ is an SOS-convex semi-algebraic function.
In addition, by simple calculations, we see that
$\Omega_2=\{(y_{1}^{2},y_{2}^{2})\in\mathbb{R}^{2} \colon (y_{1}^{2})^{2}+(y_{2}^{2})^{2}\leq1 \}$ and
$\Gamma=\left\{(a_{1},a_{2}) \in \mathbb{R}^2 \colon -1\leq a_{1}\leq1\ , -1\leq a_{2}\leq1 \right\}.$

Let $\hat{x}=(0,0).$
Then, it can be verified that $f_{1}(\hat{x})=-1<0$ and so, the Slater condition is satisfied, which means Assumption~\ref{assumption1} is satisfied.
Moreover, by letting $(\hat{y}_{1}^{2},\hat{y}_{2}^{2})=(0,0),$  and $(\hat{a}_{1},\hat{a}_{2})=(0,0),$ we see that Assumption~\ref{assumption2} also holds.
			
We now consider the following sum of squares relaxation dual problem for the problem~\eqref{FP1}:
\begin{align*}
\sup\limits_{\substack{\lambda_{0}^{i}, \lambda_{j}^{i}}} \ \ & \lambda_{0}^{3}  \\
\textrm{s.t.}\quad &\lambda_{0}^{3}h_{0}^{3}+\sum^{2}_{j=1}\lambda_{j}^{3}h_{j}^{3}+\lambda_{0}^{1}h_{0}^{1} +\lambda_{0}^{2}h_{0}^{2}+\sum^{2}_{j=1}\lambda_{j}^{2}h_{j}^{2}\in\Sigma^{2}_{2},\\
				& \left(\begin{array}{ccc} \lambda_{0}^{2} & 0 & \lambda_{1}^{2} \\  0 & \lambda_{0}^{2} & \lambda_{2}^{2} \\ \lambda_{1}^{2} & \lambda_{2}^{2} & \lambda_{0}^{2} \\ \end{array}\right)\succeq0,\\
	&{\rm diag}\left(\lambda_{0}^{3}-\lambda_{1}^{3},\  \lambda_{0}^{3}+\lambda_{1}^{3},\  \lambda_{0}^{3}-\lambda_{2}^{3},\ \lambda_{0}^{3}+\lambda_{2}^{3}\right)\succeq0,\\
				&\lambda_{0}^{2}=1, \ \lambda_{0}^{i}\geq0,\ i=1,3,\ \lambda_{j}^{2}\in\mathbb{R}, \  \lambda_{j}^{3}\in\mathbb{R}, \  j=1,2.
\end{align*}
By using Proposition~\ref{proposition1}, there exists $X\in S_{+}^{s(1)}(=S_{+}^{3})$ such that
\begin{equation*}
\lambda_{0}^{3}h_{0}^{3}+\sum^{2}_{j=1}\lambda_{j}^{3}h_{j}^{3}+\lambda_{0}^{1}h_{0}^{1} +\lambda_{0}^{2}h_{0}^{2}+\sum^{2}_{j=1}\lambda_{j}^{2}h_{j}^{2} =\langle v_{1}(x)v_{1}(x)^{T},X\rangle, \  \forall x\in\mathbb{R}^{2},
\end{equation*}
where $v_{1}(x)=(1,x_{1},x_{2})^{T}$. 
Now, we formulate the following SDP dual problem for the problem~\eqref{FP1}:
			\begin{align}\label{SDD1}
				\sup\limits_{\substack{\lambda_{j}^{i},X}} \ \ & \lambda_{0}^{3} \tag{$\widehat{\rm Q}_{1}$}\\
				\textrm{s.t.}\ \ \,   & 4\lambda_{0}^{1}+3\lambda^{3}_{0}+\lambda^{3}_{1}+\lambda^{3}_{2}=X_{11}, \,-4\lambda_{0}^{1}+3\lambda^{2}_{1}+\lambda_{2}^{2}-\lambda_{1}^{3}=2X_{12},0=2X_{23}, \notag \\ & -2\lambda^{1}_{0}+\lambda^{2}_{1}+\lambda_{2}^{2}-\lambda_{2}^{3}=2X_{13},\,
				\lambda^{1}_{0}=X_{22},\,
				\lambda^{1}_{0}=X_{33}, \notag \\
				&\left(\begin{array}{ccc} \lambda_{0}^{2} & 0 & \lambda_{1}^{2} \\  0 & \lambda_{0}^{2} & \lambda_{2}^{2} \\ \lambda_{1}^{2} & \lambda_{2}^{2} & \lambda_{0}^{2} \\ \end{array}\right)\succeq0, \notag\\
	&{\rm diag}\left(\lambda_{0}^{3}-\lambda_{1}^{3},\  \lambda_{0}^{3}+\lambda_{1}^{3}, \  \lambda_{0}^{3}-\lambda_{2}^{3}, \ \lambda_{0}^{3}+\lambda_{2}^{3}\right)\succeq0,\notag\\
				&\lambda_{0}^{2}=1, \ \lambda_{0}^{i}\geq0,\ i=1,3,\ \lambda_{j}^{2}\in\mathbb{R}, \  \lambda_{j}^{3}\in\mathbb{R}, \  j=1,2. \notag
			\end{align}
Then, the Lagrangian dual problem of the problem~\eqref{SDD1} is stated as follows:
\begin{align}\label{SDP1}
\inf\limits_{\substack{y\in\mathbb{R}^{6}\\Z_{i}\succeq0 }} \ \ & \sum_{\alpha\in\mathbb{N}^{n}_{2d}} \langle A_{0}^{2},Z_{2}\rangle\tag{${\rm Q}_{1}$}\\
				{\rm s.t.}\quad \ \, &1-3y_{00}+\langle A_{0}^{3},Z_{3}\rangle\leq0,\notag\\
				&y_{20}+y_{02}-4y_{10}-2y_{01}+4y_{00}\leq0, \notag\\
				&3y_{10}+y_{01}+\langle A_{1}^{2},Z_{2}\rangle=0, y_{10}+y_{01}+\langle A_{2}^{2},Z_{2}\rangle=0\notag \\
				&y_{00}-y_{10}+\langle A_{1}^{3},Z_{3}\rangle=0, \ y_{00}-y_{01}+\langle A_{2}^{3},Z_{3}\rangle=0, \ \notag\\
				& \sum_{\alpha}y_{\alpha} \mathcal{B}_{\alpha}\succeq0.\notag
			\end{align}
Solving the problem \eqref{SDP1} using CVX in MATLAB, we obtain the optimal value $1.4907\approx\sqrt{20}/3$ and an optimal solution $(\overline y,\overline Z_{2},\overline Z_{3})$ of the problem~\eqref{SDP1}, where
			\begin{align*}
				\overline{y} \ &= ( 1/3, 1/3, 1/3, 1/3, 1/3,1/3 )^{T} ,\\
			\overline{Z}_{2} \ & \approx\left(
			\begin{array}{rrr}
				 0.5963 & 0.2981 & -0.6667 \\
				0.2981 & 0.1491 & -0.3333 \\
				-0.6667 & -0.3333 & 0.7454 \\
			\end{array}
			\right), \\
		  \overline{Z}_{3} \ & \approx{\rm diag}(0.0000, 0.0000, 0.0000, 0.0000).
			\end{align*}
So, Theorem~\ref{thm2} implies that $\overline{x}=(1 ,1)$ is a solution of the problem $\eqref{FP1}.$  

In what follows, we will verify that the obtained point $\overline{x}=(1, 1)$ is an optimal solution of the problem~\eqref{FP1}. 
To this end, we consider the following non-fractional optimization problem,
\begin{align}\label{NFP1}
\min\limits_{x\in\mathbb{R}^{2}}\left\{ \left\| Q^\frac{1}{2}x  \right \|_2 +\gamma\sup_{( a_1,a_2 )\in\Gamma} \left\{-a_1x_1-a_2x_2+a_1+a_2-3\right\} \colon f_{1}(x)\leq0 \right\}, \tag{${{\rm NFP}_{2}}$}
\end{align}
where $\gamma\geq0$ is a parameter.
As we have seen that 
\begin{align*}
\sup_{( a_1,a_2 )\in\Gamma} \left\{-a_1x_1-a_2x_2+a_1+a_2-3\right\}=|x_1-1|+|x_2-1|-3.
\end{align*}
Note that, by Dinkelbach lemma (see Section~\ref{Sec:1}), $\overline x=(1,1)$ is an optimal solution of the problem \eqref{FP1} if and only if $(1,1)$ is an optimal solution of the problem~\eqref{NFP1} and its optimal value is equal to zero with $\gamma = \frac{\sqrt{20}}{3}$ (see also \cite{Dinkelbach1967}).
By putting $\overline x=(1,1)$ into the problem~\eqref{NFP1}, we know its optimal value is 
\begin{align*}
\left\| Q^\frac{1}{2}\overline{x}  \right \|_2 +\frac{\sqrt{20}}{3}\sup_{( a_1,a_2 )\in\Gamma} \left\{-a_1\overline{x}_1-a_2\overline{x}_2+a_1+a_2-3\right\}=0.
\end{align*}
Remember that the problem~\eqref{NFP1} is a convex one, to verify that $\overline x=(1, 1)$ is an optimal solution of the problem~\eqref{NFP1}, we can apply it to Karush--Kuhn--Tucker (KKT) optimality conditions for convex optimization problems. For $\overline x=(1,1),$ the KKT system is as follows,
\begin{equation}\label{ex2_aaa}
\left\{
  \begin{array}{lll}
    {\rm{\textbf{0}}}&\in&\partial f_2(\overline x)+\gamma\partial f_3(\overline x)+\lambda_1\partial f_1(\overline x)\\
    0&=&\lambda_1f_1(\overline x_1,\overline x_2),\ \lambda_1\geq0,
  \end{array}
\right.
\end{equation}
equivalently,
\begin{align}\label{ex2_rel1}
\left\{
  \begin{array}{lll}
    &-\frac{9}{10}\in[-1,1],\\
    &\frac{11\sqrt{20}}{60}\leq\lambda_1\leq \frac{31\sqrt{20}}{60},
  \end{array}
\right.
\end{align}
which implies that $\lambda_1>0$ does exist, and 
let $\overline\lambda_1 > 0$ be any selected such that \eqref{ex2_rel1} holds.
Then, $(\overline x, \overline\lambda_1)$ satisfies the KKT system~\eqref{ex2_aaa}, and so, $\overline x=(1,1)$ is an optimal solution of the problem~\eqref{NFP1}.
Consequently, the point $\overline x=(1,1)$ as we obtained is an optimal solution of the problem~\eqref{FP1}.
\qed
}\end{example}

\section{Conclusions}\label{sect:4}
In this paper, we studied a class of nonsmooth fractional programs with SOS-convex semi-algebraic functions.
We significantly improved the work \cite{Jiao2019b} from the aspects of both the model and the solving methods.
In other words, different to the well-known Dinkelbach method, we propose a parameter-free approach for solving such a class of fractional programs due to its own special constructions.
Remarkably, this parameter-free approach allows us to extract an optimal solution and the optimal value by solving one and only one associated SDP problem.

\subsection*{Acknowledgements}
This work was supported by the Education Department of Jilin Province (no. JJKH20241405KJ), and by the Fundamental Research Funds for the Central Universities.
Jae Hyoung Lee was supported by the National Research Foundation of Korea (NRF)
grant funded by the Korea government (MSIT) (NRF-2021R1C1C2004488).

\small

\end{document}